\documentclass[11pt]{amsart}
\usepackage{amsmath,amsfonts,amssymb,amscd,amsthm,amsbsy,upref,graphicx,mathtools,longtable,mdframed,bm,xcolor,cite,setspace,soul,color,censor}
\usepackage[bookmarks]{hyperref}
\usepackage[lmargin=1in,rmargin=1in,tmargin=1in,bmargin=1in]{geometry}

\newtheorem{thm}{Theorem}[section]
\newtheorem{corollary}[thm]{Corollary}
\newtheorem{lemma}[thm]{Lemma}

\theoremstyle{definition}

\theoremstyle{remark}
\newtheorem{remark}[thm]{Remark}
\newtheorem{assumption}[equation]{Assumption}

\numberwithin{equation}{section}
\numberwithin{thm}{section}

\newcommand{\laplacian}{\Delta}
\newcommand{\grad}{\nabla} 
\renewcommand{\eqref}[1]{\textup{{\normalfont(\ref{#1}}\normalfont)}}
\newcommand{\norm}[1]{\left\lVert#1\right\rVert}
\newcommand{\im}{\mbox{Im}\,}
\renewcommand{\epsilon}{\varepsilon}
\newcommand{\massenergy}[1]{E[#1] \norm{ #1}_2^{2 \frac{(1-s_c)}{s_c}}}
\newcommand{\masskineticenergy}[1]{\norm{\grad #1}_2^2 \norm{ #1}_2^{2 \frac{(1-s_c)}{s_c}}}
\newcommand{\masskatoenergy}[1]{E_K[#1] \norm{ #1}_2^{2 \frac{(1-s_c)}{s_c}}}
\newcommand{\masskatokineticenergy}[1]{\norm{\grad \left|#1\right|\,}_2^2 \norm{ #1}_2^{2 \frac{(1-s_c)}{s_c}}}



\begin{document}

\title[]{Scattering and Blow-Up in Both Time Directions Above the Ground State for the Focusing Nonlinear Schr{\"{o}}dinger Equation}

\author[I. Miller]{Ian Miller}
\address{Department of Mathematics, University of Colorado Boulder}
\email{ian.d.miller@colorado.edu}

\begin{abstract}
In this article we obtain new scattering and blow-up solutions for intercritical focusing nonlinear Schr{\"{o}}dinger equations (NLS) above the ground state mass-energy threshold. The main focus of this article is the establishment of some solutions with arbitrarily large mass-energy which scatter in both time directions. In particular, large mass-energy which does not arise from the Galilean transformation. We additionally obtain new criteria for blow-up in both time directions, as well as improved sufficient conditions for scattering and blow-up in one time direction.
\end{abstract}

\maketitle

\section{Introduction}\label{sec:introduction}
\par We consider the intercritical, focusing, nonlinear Schr{\"{o}}dinger equation,
\begin{equation}\tag{NLS}\label{eq:nls}
\begin{cases} & i u_t + \laplacian u = - |u|^{p-1} u,\\ & u(x,t) = u_0(x),
\end{cases}
\end{equation}
with unknown $u:\mathbb{R}^N \times \mathbb{R} \rightarrow \mathbb{C}$. The equation has scaling invariance
\[u \mapsto \lambda^{2/(p-1)} u(\lambda x, \lambda^2 t),\]
and we consider $N \ge 2$ \footnote{We establish new control of solutions based on angular momentum. In the case $N=1$, angular momentum is degenerate.} and $p$ such that this scaling is mass supercritical and energy subcritical. That is, defining
\[ s_c = \frac{N}{2} - \frac{2}{p-1},\]
the scaling is critical in $\dot{H}^{s_c}(\mathbb{R}^N)$ with
\[0 < s_c <1.\]
\par We refer to solutions of NLS simply as solutions, and explicitly state when we discuss solutions of different equations.
\par Here, for initial data lying below a certain threshold, global dynamics are very well understood. Above this threshold, it is known that global dynamics are more complex, but apart from this fact, much remains unknown. We will elaborate on the difficulties and overview existing results below, after we set up some necessary notation.
\par In this article, all function spaces will be spatial, over $\mathbb{R}^N$. We will omit the base space from description of our function spaces, for example writing $L^q$ for $L^q(\mathbb{R}^N)$, and we simplify notation for norms, writing $\norm{\cdot}_q$ for $\norm{\cdot}_{L^q}$.
\par We consider $H^1$ initial data, for which there is extensive classical theory (e.g. \cite{SulemSulem99,Cazenave03}). For initial data $u_0 \in H^1$, the equation is locally well-posed. Given a maximal interval of existence, $(T_*,T^*)$, for a solution $u$, $T^*$ is finite if and only if $\lim_{\,t \,\uparrow \,T^{*}} \norm{\grad u(t)}_2^2 = \infty$, in which case we say $u$ blows-up in finite positive time. Similarly, $T_*$ is finite if and only if $\lim_{\,t\, \downarrow\, T_*} \norm{\grad u(t)}_2^2 = \infty$, in which case we say $u$ blows-up in finite negative time.
\par Furthermore, a solution $u$ satisfies several conservation laws on the full interval of existence. These include
\begin{align*}
\mbox{Mass:} \quad &\norm{u(t)}_2^2 = \norm{u_0}_2^2,\\
\mbox{Energy:} \quad &E[u](t) \coloneqq \frac{1}{2} \norm{\grad u(t)}_2^2 - \frac{1}{p+1} \norm{u(t)}_{p+1}^{p+1} = E[u_0],\\
\mbox{Linear Momentum:} \quad &\textbf{P}[u](t) \coloneqq \int\mathcal{P}[u(x,t)]\, dx = \textbf{P}[u_0],\\
\mbox{Angular Momentum:} \quad &\textbf{M}[u](t) \coloneqq \int \mathcal{P}[u(x,t)] \wedge \textbf{x}\, dx = \textbf{M}[u_0],
\end{align*}
where
\[
\mathcal{P}[u] \coloneqq 2 \, \im \left[\overline{u}\, \grad u \right]
\]
is the momentum density.\\
\par In our definition of angular momentum, the wedge of vectors
\[\langle v_1, \ldots , v_N \rangle \wedge \langle w_1, \ldots , w_N \rangle\]
can be thought of as a vector with ${ N \choose 2}$ components, given by
\[ v_j w_k - v_k w_j\]
for $j<k$. In the case of 3 dimensions, the angular momentum is typically defined by the vector
\[
\textbf{M}^*[u](t) \coloneqq - \int \textbf{x} \times \mathcal{P}[u(x,t)]\,dx.
\]
There is an equivalence between these definitions. In particular, the vector $\textbf{M}^*[u](t)$ has the same components as $\textbf{M}[u](t)$ up to a sign, and in 3 dimensions, our results can be formulated with either definition with no change to the statement of the theorems.
\par Conservation of angular momentum in general dimension follows component-wise in the same manner as for three dimensions.

\subsection{The Soliton Solution}
\par As shown in \cite{Weinstein82} (see also \cite{HolmerRoudenko07}), NLS admits a soliton solution
\[u_Q(x,t) = e^{i \lambda t} Q(\alpha x)\]
where $Q:\mathbb{R}^N \rightarrow \mathbb{R}_{\ge 0}$ is the ground state solution for the equation
\[\frac{N(p-1)}{4} \laplacian Q - \left(1 - \frac{(N-2)(p-1)}{4} \right) Q + Q^p = 0,\]
$\alpha = \frac{\sqrt{N(p-1)}}{2}$, and $\lambda = 1- \frac{(N-2)(p-1)}{4}$. The Gagliardo-Nirenberg inequality,
\[
\norm{u}_{p+1}^{p+1} \le C_{GN} \norm{\grad u}_2^{\frac{N(p-1)}{2}} \norm{u}_2^{2 - \frac{(N-2)(p-1)}{2}},
\]
which holds for $u \in H^1$, is optimized by $Q$, and has sharp constant
\[
C_{GN} = \frac{p+1}{2 \norm{Q}_2^{p-1}}.
\]
Furthermore, $Q$ decays exponentially (see \cite{BerestyckiLions83,Tao06}).
By construction, $Q$ satisfies
\[
\norm{\grad Q}_2^2 = \norm{Q}_2^2,
\]
and this together with Pohozaev's identity implies
\[
\norm{\grad Q}_2^2 = \frac{2}{p+1} \norm{Q}_{p+1}^{p+1}.
\]
\par The soliton $u_Q$ also optimizes the Gagliardo-Nirenberg inequality, so that
\[
\norm{u_Q}_{p+1}^{p+1} = C_{GN} \norm{\grad u_Q}_2^{\frac{N(p-1)}{2}} \norm{u_Q}_2^{2 - \frac{(N-2)(p-1)}{2}},
\]
and by computation using the identities of $Q$, we obtain identities
\begin{align}
\label{eq:soliton-identity1} E[u_Q] &= \frac{N(p-1)-4}{2N(p-1)} \norm{\grad u_Q}_2^2,\\
\label{eq:soliton-identity2} C_{GN} &= \frac{2(p+1)}{N(p-1)\left(\norm{\grad u_Q}_2^{s_c}\norm{u_Q}_2^{1-s_c}\right)^{p-1}}.
\end{align}

\subsection{Variance and Other Important Quantities}
\par At a given time $t$, solutions have spatial variance
\[V[u](t) \coloneqq \int |x|^2 |u(x,t)|^2\, dx.\]
If $V[u_0]$ is finite, then $V[u](t)$ is finite for the full interval of existence and satisfies the well-known identities \cite{SulemSulem99}
\begin{align*}
V_t[u](t) &= 2 \int \textbf{x} \cdot \ \mathcal{P}\left[u(x,t)\right] \,dx,\\
V_{tt}[u](t) &= 8 \norm{\grad u(t)}_2^2 - \frac{4N(p-1)}{p+1} \norm{u(t)}_{p+1}^{p+1},
\end{align*}
where the identity for $V_{tt}[u](t)$ is the classical virial identity.
\par By either direct computation using properties of $u_Q$ stated above, or else the stationary properties of the soliton,
\begin{equation}\label{eq:variance-value}
8 \norm{\grad u_Q}_2^2 - \frac{4N(p-1)}{p+1} \norm{u_Q}_{p+1}^{p+1}=0.
\end{equation}
\par In what follows, given initial data $u_0 \in H^1$, expressions such as $V_t[u_0]$ will refer to $V_t[u](0)$ where $u$ is the solution with initial data $u_0$.

\par The quantity $\norm{\grad u (t)}_2^2$ will be referred to as kinetic energy. We also consider the quantity $\norm{\grad |u|(t)}_2^2$, which will be referred to as Kato kinetic energy. An identity first noted in \cite{CzubakMillerRoudenko24} and presented in Lemma \ref{lem:remainder} demonstrates an aspect of the connection between Kato kinetic energy and momentum. This connection has been exploited to obtain blow-up results for electromagnetic nonlinear Schr{\"{o}}dinger equations \cite{CzubakMillerRoudenko24}, and will be used in our main result. Also significant will be the Kato energy of a solution, which is defined by
\[
E_K[u](t) \coloneqq \frac{1}{2} \norm{\grad \left|u(t)\right| \,}_2^2 - \frac{1}{p+1} \norm{u(t)}_{p+1}^{p+1},
\]
and is not a conserved quantity.
\par Kato kinetic energy and Kato energy are introduced in \cite{CzubakMillerRoudenko24} for the electromagnetic nonlinear Schr{\"{o}}dinger equation, and we do not know of its presence in prior literature for NLS. However, as we find in this article, Kato kinetic energy reveals important structure regarding the mass profile, $|u_0|$, of initial data, and we will contextualize many previous results in terms of Kato energy and Kato kinetic energy. For existing results, this contextualization is a natural alternative to the originally presented context, and does not modify the results, as is discussed further in Remark \ref{rem:grad-comp}.
\subsection{At and Below the Ground State}
\par Questions of global existence and blow-up at and below the ground state have been extensively studied (e.g. \cite{GinibreVelo79-1,CazenaveWeissler92,HolmerRoudenko07,DuyckaertsHolmerRoudenko08,HolmerRoudenko08,HolmerRoudenko10,Guevara14,GuevaraCarreon12,CamposFarahRoudenko22,DuyckaertsRoudenko10}). We review certain results which are relevant to our work and are established, in their current generality, in \cite{Guevara14} and \cite{CamposFarahRoudenko22}.
\par With respect to the scale invariant quantities
\begin{align*}
\mbox{Mass-Energy}:& \quad \massenergy{u(t)},\\
\mbox{Mass-Kinetic Energy}:& \quad \masskineticenergy{u(t)},
\end{align*}
the soliton $u_Q$ provides a threshold for global existence and blow-up. We say that initial data $u_0$ is below the ground state if $u_0$ has smaller mass-energy than the soliton $u_Q$, so
\[E[u_0]\norm{u_0}_2^{2\frac{(1-s_c)}{s_c}}  < E[u_Q] \norm{u_Q}_2^{2\frac{(1-s_c)}{s_c}}.\]
\par The study of initial data at and below the ground state has yielded very satisfactory results. Precursors to global results for the intercritical equation appear in works such as \cite{KenigMerle06,Weinstein82} which study mass and energy critical equations. Holmer and Roudenko began the investigation of the intercritical equation in \cite{HolmerRoudenko07}. The full resolution for the case of finite variance,\footnote{The case of infinite variance, while having well developed theory, is to some extent not fully resolved. This is due to remaining questions regarding an a priori infinite time blow-up scenario (see \cite{HolmerRoudenko10}).} was achieved in following theorems of Guevara \cite{Guevara14}, and Campos, Farah, and Roudenko \cite{CamposFarahRoudenko22}.
\begin{thm}[Below the Ground State \cite{Guevara14}]\label{thm:intro-ref-below}
Assume, $0<s_c<1$, and $u_0 \in H^1$ with $|x|u_0 \in L^2$ such that
\[
\massenergy{u_0} < \massenergy{u_Q}.
\]
\begin{itemize}
\item If
\[
\masskineticenergy{u_0}<\masskineticenergy{u_Q},
\]
then the solution $u$ exists globally and scatters in both time directions.
\item If
\[
\masskineticenergy{u_0}>\masskineticenergy{u_Q},
\]
then the solution $u$ blows-up in finite time in both time directions.
\end{itemize}
\end{thm}
\begin{remark}
\par The energy assumption implies that $u_0$ can not satisfy
\[
\masskineticenergy{u_0} = \masskineticenergy{u_Q}.
\]
\end{remark}
\par The above theorem shows that for finite variance initial data below the ground state, there is a true dichotomy with global and blow-up solutions classified by mass-kinetic energy.
\begin{thm}[`Threshold Solutions': At the Ground State \cite{CamposFarahRoudenko22}]\label{thm:intro-ref-at} Assume $0<s_c<1$. The theory of solutions at the ground state can be presented in two parts.\\
\par \textbf{Part 1:}
\par There exist radial solutions $u_Q^+$ and $u_Q^-$ to NLS in $H^1$ satisfying the following properties
\begin{itemize}
\item $\norm{u_Q^-(t)}_2^2 = \norm{u_Q^+(t)}_2^2 = \norm{u_Q}_2^2$, \quad $E[u_Q^-](t) = E[u_Q^+](t) = E[u_Q]$,
\item $\norm{\grad u_Q^-(0)}_2^2 < \norm{\grad u_Q}_2^2 < \norm{ \grad u_Q^+(0)}_2^2 $,
\item In positive time, $u_Q^-$ converges to the soliton $u_Q$ in the sense that there are $C,\alpha>0$ such that
\[
\norm{u_Q^-(t) - u_Q}_{H^1} \le C e^{-\alpha t},
\]
and in negative time $u_Q^-$ exists globally and scatters.
\item In positive time, $u_Q^+$ converges to the soliton $u_Q$ in the sense that there are $C,\alpha>0$ such that
\[
\norm{u_Q^+(t) - u_Q}_{H^1} \le C e^{-\alpha t},
\]
and $u_Q^+$ blows-up in finite negative time.
\end{itemize}
\par \textbf{Part 2:}
\par Suppose $u_0 \in H^1$ with $|x|u_0 \in L^2$ satisfies
\[
\massenergy{u_0} = \massenergy{u_Q}.
\]
\begin{itemize}
\item If
\[\masskineticenergy{u_0}<\masskineticenergy{u_Q},
\]
then either
\begin{itemize}
\item the solution satisfies $u=u_Q^-$ up to symmetry, or
\item the solution exists globally and scatters in both time directions.
\end{itemize}
\item If
\[\masskineticenergy{u_0}=\masskineticenergy{u_Q},
\]
then the solution satisfies $u=u_Q$ up to symmetries.
\item If
\[\masskineticenergy{u_0}>\masskineticenergy{u_Q},
\]
then either
\begin{itemize}
\item the solution satisfies $u=u_Q^+$ up to symmetry, or
\item the solution blows-up in finite time in both time directions.
\end{itemize}
\end{itemize}
\end{thm}
\par This theorem shows that once the threshold is reached, the dichotomy seen under the ground state begins to break down, where in addition to the soliton, there is a single special solution as an exception to scattering and a single special solution as an exception to blow-up. This leaves five possibilities for behavior of solutions at the ground state. Later, we will see that above the ground state the classification property initial data below the ground state enjoys breaks down even further.

\subsection{Above the Ground State}
\par Note that at and below the ground state, for the case of finite variance, the theorems of Guevara \cite{Guevara14}, and of Campos, Fara, and Roudenko \cite{CamposFarahRoudenko22}, completed the classification of initial data based on general global behavior of solutions.
\par Above the ground state, there is considerably less known. Here we will review relevant results.\footnote{Some of these results are proved for specific NLS equations, commonly the 3 dimensional cubic equation. At times these restrictions seem to be for convenience, with arguments generalizing to more general intercritical equations, though there are times when these restrictions are important for arguments. Here we wish to review what has been done for any case of intercritical equation, so for discussion of background results we will avoid mention of such restrictions.}
\par We begin with a remark concerning linear momentum. Given $v \in \mathbb{R}^N$, the Galilean transformation
\[
u(x,t) \mapsto u_v(x,t) = e^{ix \cdot v}e^{-it|v|^2} u(x-2vt,t)
\]
allows us to transform between inertial reference frames, with $u_v$ being a solution if $u$ is, and with
\[
\textbf{P}[u_v] = \textbf{P}[u] + 2v.
\]
\par For nontrivial $u$, the Galilean transformation preserves $\norm{u(t)}_2^2$, and by choosing $v$ with large enough magnitude, $E[u_v](t)$ can be taken to be arbitrarily large. This allows for classification of global behavior for certain solutions arbitrarily far above the ground state by taking initial data below the ground state and performing an appropriate Galilean transformation. In this article, we only consider results which are nontrivial under the assumption on linear momentum $\textbf{P}[u_0] =0$.
\par As stated above, there are existing results which indicate that the problem above the ground state is more delicate than below the ground state. Nakanishi and Schlag \cite{NakanishiSchlag12} identify at least nine different possibilities for global dynamics, each attained by infinitely many initial data, including scattering in one time direction and blow-up in the other (see also \cite{CazenaveWeissler92,DuyckaertsRoudenko15}, for explicit examples).
\par It is natural to separately consider results for initial data which is close to the ground state, and results where initial data is allowed to be far from the ground state.
\par There are two results we know of imposing some restriction on distance of initial data from the ground state. Nakanishi and Schlag \cite{NakanishiSchlag12} obtain results for initial data which is close to the ground state in mass-energy. Beceanu \cite{Beceanu08} obtains obtains results for initial data which is close to the ground state in norm.
\par We now review what is known for initial data permitted to be arbitrarily far from the ground state. Given initial data $u_0 \in H^1$ with $|x|u_0 \in L^2$, a rewriting of the standard virial identity blow-up argument for positive energy implies that if
\begin{equation}\label{eq:negative-virial-ext}
 \left[ E[u_0] - \frac{|V_t[u_0]|^2}{32 V[u_0]}\right] \norm{u_0}_2^{2\frac{(1-s_c)}{s_c}} \le 0,
\end{equation}
and
\[
V_t[u_0] <0,
\]
then the solution $u$ with initial data $u_0$ blows-up in finite positive time (see \cite[Theorem 4.2]{Weinstein82},\cite{SulemSulem99}). Duyckaerts and Roudenko \cite{DuyckaertsRoudenko15} improve on this by establishing results after raising the right hand side of the inequality \eqref{eq:negative-virial-ext} from $0$ to the ground state. These results include the establishment of both solutions which are global in positive time, and solutions which blow-up in finite positive time. We now recount this result in the intercritical case.
\begin{thm}[\cite{DuyckaertsRoudenko15}]\label{thm:intro-ref-DR}
Suppose $0<s_c<1$, and consider $u_0 \in H^1$ with $|x|u_0 \in L^2$. Furthermore, assume
\begin{equation}\label{eq:intro-energy-above}
 \left[ E[u_0] - \frac{|V_t[u_0]|^2}{32 V[u_0]}\right] \norm{u_0}_2^{2\frac{(1-s_c)}{s_c}} \le  E[u_Q] \norm{u_Q}_2^{2\frac{(1-s_c)}{s_c}}.
\end{equation}
\begin{itemize}
\item (Existence and Scattering) If
\begin{equation}\label{eq:intro-kinetic-ass-below}
\masskatokineticenergy{u_0} < \masskineticenergy{u_Q},
\end{equation}
and
\[
V_t(0) \ge 0,
\]
then the solution $u$ with initial data $u_0$ exists globally forward in time and scatters forward in time.
\item (Blow-Up) If
\begin{equation}\label{eq:intro-kinetic-ass-above}
\masskatokineticenergy{u_0} > \masskineticenergy{u_Q},
\end{equation}
and
\[
V_t(0) \le 0,
\]
then the solution $u$ with initial data $u_0$ blows-up in finite positive time.
\end{itemize}
\end{thm}
\begin{remark}\label{rem:grad-comp}
This result is formulated somewhat differently, but equivalently, in its original paper. In conditions \eqref{eq:intro-kinetic-ass-below} and \eqref{eq:intro-kinetic-ass-above}, the original paper considers the quantity $\norm{\cdot}_{p+1}^{p+1}$ where we consider Kato kinetic energy. By a computation (see \cite[`Double Cut Lemma']{CzubakMillerRoudenko24} as well as the proofs of Corollary \ref{cor:scattering} or Theorem \ref{thm:main-unidirectional} in this article), if $w \in H^1$ with $\masskatoenergy{w} \le \massenergy{u_Q}$, then
\begin{equation}\label{eq:intro-kinetic-below}
\norm{w}_{p+1}^{p+1} \norm{w}_2^{2 \frac{(1-s_c)}{s_c}} < \norm{u_Q}_{p+1}^{p+1} \norm{u_Q}_2^{2 \frac{(1-s_c)}{s_c}}, \quad \mbox{and} \quad \masskatokineticenergy{w} < \masskineticenergy{u_Q}
\end{equation}
are equivalent, and
\begin{equation}\label{eq:intro-kinetic-above}
\norm{w}_{p+1}^{p+1} \norm{w}_2^{2 \frac{(1-s_c)}{s_c}} > \norm{u_Q}_{p+1}^{p+1} \norm{u_Q}_2^{2 \frac{(1-s_c)}{s_c}}, \mbox{ and } \masskatokineticenergy{w} > \masskineticenergy{u_Q}
\end{equation}\\
are equivalent. In fact, if $\massenergy{w} \le \massenergy{u_Q}$, then all of\\
\begin{align*}
\norm{w}_{p+1}^{p+1} \norm{w}_2^{2 \frac{(1-s_c)}{s_c}} &< \norm{u_Q}_{p+1}^{p+1} \norm{u_Q}_2^{2 \frac{(1-s_c)}{s_c}},\\
\masskatokineticenergy{w} &< \masskineticenergy{u_Q},\\
\masskineticenergy{w} &< \masskineticenergy{u_Q},
\end{align*}
are equivalent, and likewise for the reverse inequalities.
\par The energy assumption \eqref{eq:intro-energy-above} implies, using Lemma \ref{lem:momentum-inequality} in this article, that $\masskatoenergy{u_0} \le \massenergy{u_Q}$, and hence the equivalences \eqref{eq:intro-kinetic-below} and \eqref{eq:intro-kinetic-above} apply. This also relates to the connection between below ground state and above ground state results as noted in \cite{DuyckaertsRoudenko15}.
\end{remark}
\begin{remark}
The improvement that Theorem \ref{thm:intro-ref-DR} makes on the positive energy virial identity result, improving \eqref{eq:negative-virial-ext} to \eqref{eq:intro-energy-above}, has a parallel structure to the improvement that Theorems \ref{thm:intro-ref-below} and \ref{thm:intro-ref-at} jointly make by raising the negative energy blow-up condition
\[
\massenergy{u_0}<0
\]
to
\[
\massenergy{u_0} \le \massenergy{u_Q}.
\]
\end{remark}
\par Theorem \ref{thm:intro-ref-DR} provides the most important context for the main global existence and blow-up results of this article. For our result in one time direction, Theorem \ref{thm:main-unidirectional}, we extend Theorem \ref{thm:intro-ref-DR} by obtaining results after further weakening \eqref{eq:intro-energy-above} to
\begin{equation}\notag
 \left[ E[u_0] - \frac{|V_t[u_0]|^2 + |2 \textbf{M}[u_0]|^2}{32 V[u_0]}\right] \norm{u_0}_2^{2\frac{(1-s_c)}{s_c}} \le  E[u_Q] \norm{u_Q}_2^{2\frac{(1-s_c)}{s_c}}.
\end{equation}
\par Incorporation of the term
\begin{equation}\label{eq:angular-term}
\frac{|2 \textbf{M}[u_0]|^2}{32 V[u_0]} \norm{u_0}_2^{2\frac{(1-s_c)}{s_c}}
\end{equation}
causes this inequality to provide a weaker condition, which allows us to establish results for initial data which is, in a sense, further from the ground state.
\par In fact, we find that the term \eqref{eq:angular-term} provides a control of solutions which is less time dependent than
\begin{equation}\notag
\frac{|V_t[u_0]|^2}{32 V[u_0]} \norm{u_0}_2^{2\frac{(1-s_c)}{s_c}},
\end{equation}
especially for the case of global existence. In Theorem \ref{thm:main-bidirectional}, consideration of the term \eqref{eq:angular-term} on its own produces a global existence result which is a direct extension of global existence below the ground state, weakening the energy assumption from
\begin{equation}\notag
E[u_0] \norm{u_0}_2^{2\frac{(1-s_c)}{s_c}} <  E[u_Q] \norm{u_Q}_2^{2\frac{(1-s_c)}{s_c}},
\end{equation}
to
\begin{equation}\label{eq:energy-angular-intro-1}
 \left[ E[u_0] - \frac{|2 \textbf{M}[u_0]|^2}{32 V[u_0]}\right] \norm{u_0}_2^{2\frac{(1-s_c)}{s_c}} <  E[u_Q] \norm{u_Q}_2^{2\frac{(1-s_c)}{s_c}},
\end{equation}
with no further restrictions. In particular, this is a global result which has no dependence on time direction, and hence attains the goal of this article. In addition, this global result places no restriction on $V_t[u_0]$. We also use the energy condition \eqref{eq:energy-angular-intro-1} to obtain a new result concerning blow-up in two time directions.
\par For initial data far from the ground state, in addition to the results of \cite{DuyckaertsRoudenko15} discussed above, there have been a few other results for blow-up, as well as one other result concerning global existence.
\par Both \cite{DuyckaertsRoudenko15} and \cite{HolmerPlatteRoudenko10} present new criteria for blow-up. In \cite{HolmerPlatteRoudenko10}, Holmer, Platte, and Roudenko also construct new blow-up solutions for the 3d cubic equation using numerical computations, which shows that there is initial data with mass-Kato kinetic energy below the ground state which blows-up in at least one time direction. The only other result for initial data above the ground state regards global existence, and can be found in \cite{CazenaveWeissler92}, where Cazenave and Weissler demonstrate a phase modulation which allows for construction of solutions of arbitrarily large mass-energy (and arbitrarily large mass-Kato kinetic energy) which exist globally in at least one time direction.
\par The result of \cite{CazenaveWeissler92} demonstrates that there is initial data of any mass-Kato kinetic energy which exists and even scatters forward in time. One can ask if there is a similar result regarding blow-up. Surprisingly, even though reversing the phase modulation from \cite{CazenaveWeissler92} produces attractive candidates for finite time blow-up, there has only been a partial result in this direction, which is the construction of blow-up solutions found in  \cite{HolmerPlatteRoudenko10}.
\subsection{Main Results}
\par We are now prepared to state our main theorems. We begin with our result regarding global existence and blow-up in both time directions.
\begin{thm}\label{thm:main-bidirectional}
Suppose $N \ge 2$, and $0<s_c<1$, and consider initial data $u_0 \in H^1$ with $|x| u_0 \in L^2$, such that
\begin{equation}\label{eq:bidirectional-lower-energy-assumption}
\massenergy{u_0}>\massenergy{u_Q},
\end{equation}
and
\begin{equation}\label{eq:bidirectional-energy-assumption}
\left[ E[u_0] - \frac{|2 \textbf{M}[u_0]|^2}{32 V[u_0]}\right] \norm{u_0}_2^{2\frac{(1-s_c)}{s_c}} \le E[u_Q]\norm{u_Q}_2^{2\frac{(1-s_c)}{s_c}}.
\end{equation}
\begin{itemize}
    \item (Existence) If
    \begin{equation}\label{eq:kinetic-existence-bi}
    \norm{\grad |u_0|}_2^{2} \norm{u_0}_2^{2\frac{(1-s_c)}{s_c}} < \norm{\grad u_Q}_2^{2} \norm{u_Q}_2^{2\frac{(1-s_c)}{s_c}}
    \end{equation}
    then the solution $u$ with initial data $u_0$ exists for all time and scatters in $H^1$ in both time directions.
    \item (Blow-Up) If
    \begin{equation}\label{eq:kinetic-blow-up-bi}
    \norm{\grad |u_0|}_2^{2}\norm{u_0}_2^{2\frac{(1-s_c)}{s_c}} > \norm{\grad u_Q}_2^{2}\norm{u_Q}_2^{2\frac{(1-s_c)}{s_c}},
    \end{equation}
    and either,
    \begin{itemize}
    \item $V_t[u_0] = 0$ and
    \begin{align}\label{eq:bidirectional-nonlinear-condition-1}
    \frac{|2\textbf{M}[u_0]|^2}{32 V[u_0]}\norm{u_0}_2^{2\frac{(1-s_c)}{s_c}} &> P^* \left(\frac{8}{N(p-1)-4}\left(E[u_0]\norm{u_0}_2^{2\frac{(1-s_c)}{s_c}} - E[u_Q]\norm{u_Q}_2^{2\frac{(1-s_c)}{s_c}} \right)\right)\\
    &\quad + \left(E[u_0]\norm{u_0}_2^{2\frac{(1-s_c)}{s_c}} - E[u_Q]\norm{u_Q}_2^{2\frac{(1-s_c)}{s_c}} \right),\notag
    \end{align}
    where $P^*(x)$ is defined in Lemma \ref{lem:nonlinear-ineq}, or
    \item for some $\lambda>1$ small enough so that
    \[
    \left[ E[u_0] - \frac{|2 \textbf{M}[u_0]|^2}{32 \lambda V[u_0]}\right] \norm{u_0}_2^{2\frac{(1-s_c)}{s_c}} < E[u_Q]\norm{u_Q}_2^{2\frac{(1-s_c)}{s_c}},
    \]
    $u_0$ satisfies the nonlinear condition
    \begin{align}\label{eq:bidirectional-nonlinear-condition-2}
    \frac{|2\textbf{M}[u_0]|^2}{32 V[u_0]}\norm{u_0}_2^{2\frac{(1-s_c)}{s_c}} &> \lambda\, G\left(E[u_0]\norm{u_0}_2^{2\frac{(1-s_c)}{s_c}} - E[u_Q]\norm{u_Q}_2^{2\frac{(1-s_c)}{s_c}}\, ,\, \frac{|V_t[0]|^2}{(\lambda-1) V[0]}  \norm{u_0}_2^{2\frac{(1-s_c)}{s_c}}\right)
    \end{align}
    where
    \[
    G(x,y) = P^*\left(\frac{1}{2N(p-1)-8} \left(16 x+y\right)\right) + x,
    \]
    \end{itemize}
    then the solution $u$ blows-up in finite time in both time directions.
\end{itemize}
\end{thm}
\begin{remark}
There exist infinitely many initial data with $\massenergy{u_0}$ arbitrarily large and with $\frac{|2 \textbf{M}[u_0]|^2}{32 V[u_0]} \norm{u_0}_2^{2\frac{(1-s_c)}{s_c}}$ also sufficiently large for $u_0$ to satisfy the hypotheses of Theorem \ref{thm:main-bidirectional}.
\end{remark}
\begin{remark}
There exists scattering theory for initial data which is small in the critical norm. One can construct data which is small in critical norm and has large mass-energy and mass-kinetic energy by choosing suitable $\widehat{u_0}$. Our result and previous work above the ground state (e.g. \cite{DuyckaertsRoudenko15}) do not require smallness of critical norm and the bound on mass-Kato kinetic energy does not impose a bound on the critical norm.
\end{remark}
\par For discussion of the function $P^*$ and conditions \eqref{eq:bidirectional-nonlinear-condition-1} and \eqref{eq:bidirectional-nonlinear-condition-2}, see Remark \ref{rem:comp-1} and Lemma \ref{lem:nonlinear-ineq}. We also obtain the following result for behavior in one time direction.
\begin{thm}\label{thm:main-unidirectional}
Suppose $N \ge 2$, and $0<s_c<1$, and consider initial data $u_0 \in H^1$ with $|x| u_0 \in L^2$, such that
\begin{equation}\label{eq:unidirectional-lower-energy-assumption}
\massenergy{u_0}>\massenergy{u_Q},
\end{equation}
and
\begin{equation}\label{eq:unidirectional-energy-assumption}
 \left[ E[u_0] - \frac{|V_t[u_0]|^2 + |2 \textbf{M}[u_0]|^2}{32 V[u_0]}\right] \norm{u_0}_2^{2\frac{(1-s_c)}{s_c}} \le  E[u_Q] \norm{u_Q}_2^{2\frac{(1-s_c)}{s_c}}.
\end{equation}

\begin{itemize}
    \item (Existence) If
    \begin{equation}\label{eq:kinetic-existence-uni}
    \norm{\grad |u_0|}_2^{2} \norm{u_0}_2^{2\frac{(1-s_c)}{s_c}} < \norm{\grad u_Q}_2^{2} \norm{u_Q}_2^{2\frac{(1-s_c)}{s_c}}
    \end{equation}
    and
    \begin{equation}\label{eq:variance-existence-uni}
    V_t[u_0] \ge 0,
    \end{equation}
    then the solution $u$ with initial data $u_0$ exists for all positive time and scatters in $H^1$ forward in time.
    \item (Blow-Up) If
    \begin{equation}\label{eq:kinetic-blow-up-uni}
    \norm{\grad |u_0|}_2^{2 } \norm{u_0}_2^{2\frac{(1-s_c)}{s_c}} > \norm{\grad u_Q}_2^{2} \norm{u_Q}_2^{2\frac{(1-s_c)}{s_c}},
    \end{equation}
    and
    \begin{equation}\label{eq:variance-blow-up-uni}
    V_t[u_0] < 0,
    \end{equation}
    then if $u_0$ satisfies the nonlinear condition
    \begin{align}\label{eq:unidirectional-nonlinear-condition}
    \frac{|2\textbf{M}[u_0]|^2}{32 V[u_0]}\norm{u_0}_2^{2\frac{(1-s_c)}{s_c}} > F \left( \left(E[u_0] - \frac{|V_t[u_0]|^2}{32 V[u_0]}\right)\norm{u_0}_2^{2\frac{(1-s_c)}{s_c}} - E[u_Q]\norm{u_Q}_2^{2\frac{(1-s_c)}{s_c}} \right)
    \end{align}
    where
    \begin{equation}\notag
    F(x) = \begin{cases}
    x + P^*\left (\frac{8}{N(p-1)-4}\, x \right) & x \ge 0\\ 0 & x \le 0,
    \end{cases}
    \end{equation}
    and $P^*(x)$ is defined in Lemma \ref{lem:nonlinear-ineq}, then the solution $u$ with initial data $u_0$ blows-up in finite positive time.
    \end{itemize}
\end{thm}
\begin{remark}
\par When \eqref{eq:bidirectional-lower-energy-assumption} in Theorem \ref{thm:main-bidirectional} or \eqref{eq:unidirectional-lower-energy-assumption} in Theorem \ref{thm:main-unidirectional} do not hold, the initial data is at or below the ground state and, as noted above, global dynamics have been previously resolved.
\end{remark}
\begin{remark} \label{rem:comp-1}
\par In the case
\begin{equation}\label{eq:energy-comp}
 \left[ E[u_0] - \frac{|V_t[u_0]|^2}{32 V[u_0]}\right] \norm{u_0}_2^{2\frac{(1-s_c)}{s_c}} \le  E[u_Q] \norm{u_Q}_2^{2\frac{(1-s_c)}{s_c}},
\end{equation}
the assumption \eqref{eq:unidirectional-lower-energy-assumption} implies that \eqref{eq:unidirectional-nonlinear-condition} holds trivially, and Theorem \ref{thm:main-unidirectional} reduces exactly to Theorem \ref{thm:intro-ref-DR}. Hence Theorem \ref{thm:main-unidirectional} presents a new result in the case
\begin{equation}\notag
 \left[ E[u_0] - \frac{|V_t[u_0]|^2}{32 V[u_0]}\right] \norm{u_0}_2^{2\frac{(1-s_c)}{s_c}} - E[u_Q] \norm{u_Q}_2^{2\frac{(1-s_c)}{s_c}}>0.
\end{equation}
\par The nonlinear condition \eqref{eq:unidirectional-nonlinear-condition} asserts that when the failure to satisfy \eqref{eq:energy-comp} is large, in the sense that
\[
\left[ E[u_0] - \frac{|V_t[u_0]|^2}{32 V[u_0]}\right] \norm{u_0}_2^{2\frac{(1-s_c)}{s_c}} - E[u_Q] \norm{u_Q}_2^{2\frac{(1-s_c)}{s_c}}
\]
is large, then in order to obtain blow-up, the angular momentum must also be large, in the sense that
\[
\frac{|2\textbf{M}[u_0]|^2}{32V[u_0]}\norm{u_0}_2^{2\frac{(1-s_c)}{s_c}}
\]
must be large.
\par The nonlinear condition \eqref{eq:bidirectional-nonlinear-condition-1} is simply \eqref{eq:unidirectional-nonlinear-condition} with $V_t[u_0] = 0$. The condition \eqref{eq:bidirectional-nonlinear-condition-2} serves a similar function to \eqref{eq:unidirectional-nonlinear-condition}, but is adapted to a somewhat more delicate setting. In particular, while the definite sign on $V_t[u_0]$ contributes to blow-up for results in one time direction. The loss of definite sign in the two directional case makes $V_t[u_0]$ an obstruction. This obstruction accounts for the changes in the nonlinear condition.
\end{remark}
\par The global existence and blow-up results, Theorems \ref{thm:main-bidirectional} and \ref{thm:main-unidirectional}, are obtained by using bootstrap methods. The key property allowing us to overcome the breakdown of standard dichotomy arguments is that while the mass-energy of the solutions that we consider is above the ground state, the mass-Kato energy remains below the ground state. The mass-Kato energy can be thought of as mass-energy with a modification by momentum, and can be analyzed with methods analogous to the case of mass-energy.
\par Compared to the prior work for large mass-energy, the main theorem of Duyckaerts and Roudenko in \cite{DuyckaertsRoudenko15}, our main theorems take a different perspective. In particular, Duyckaerts and Roudenko are able to reduce control of \eqref{eq:intro-energy-above} and other key time dependent properties of the solution to the control of certain functions defined by a solution. In our results, we control such properties of solutions directly. Since initial data considered by Duyckaerts and Roudenko satisfies the hypothesis of Theorem \ref{thm:main-unidirectional}, our proof offers a somewhat alternate perspective for this prior result.
\par We also establish the following result of independent interest regarding the quantity $\norm{ \grad |u|}_2^2$. 
\begin{thm}\label{thm:ske-continuous}
The map from $H^1$ into $\mathbb{R} \ge 0$ given by $u \mapsto \norm{\grad |u|}_2^2$ is continuous.
\end{thm}
\par A major difficulty in establishing Theorem \ref{thm:ske-continuous} is that in general, for small $v$, one can control
\begin{equation}\label{eq:ske-local}
\left| \grad | u+v| - \grad |u| \right|
\end{equation}
by $|\grad u|$, but not by smallness of $v$. More so, one has worse control of \eqref{eq:ske-local} when $|u|$ is small.
\par To obtain this continuity result, we in essence analyze specific properties of functions $\frac{1}{|u|}$ for $u \in H^1(\mathbb{R}^N)$ by constructing a decomposition of the measure space $\mathbb{R}^N$ based on interacting dyadic decompositions of functions.

\subsection*{Acknowledgment}
This project was carried out as a part of the author's thesis work. The author would like to thank their advisor, Magdalena Czubak, for her many valuable suggestions regarding this project as well as her immense generosity and support throughout the PhD training.

\section{Kato Kinetic Energy and the Continuity Result}\label{sec:continuity-result}
We begin with the proof of Theorem \ref{thm:ske-continuous}, and follow with general discussion of Kato kinetic energy.
\subsection{Proof of Theorem \ref{thm:ske-continuous}}
\par We consider $u \in H^1$. From Kato's inequality, we have that Kato kinetic energy is bounded by kinetic energy. That is,
\[\norm{\grad |u|}_2^2 \le \norm{\grad u}_2^2.\]
Theorem \ref{thm:ske-continuous} will allow us to upgrade from a bound to continuity.\\
\par An elementary lemma will be useful for establishing this result.
\begin{lemma}\label{lem:ang-dif}
If $z,w \in \mathbb{C}$ with $z \neq 0$, and $w \neq 0$ then
\[
\left| \frac{z+w}{|z+w|} - \frac{z}{|z|} \right| \le 2\, \mbox{min} \left\{ 1, \frac{|w|}{|z|} \right\}.
\]
\end{lemma}
\par This lemma follows from the triangle inequality.
\par We are now prepared to prove Theorem \ref{thm:ske-continuous}.
\begin{proof}
\par Consider $u \in H^1$ and $\{v_n\} \subset H^1$ with $v_n \overset{H^1}{\rightarrow} 0$. We must show that
\[
\int \left|\grad\left| u+v_n \right| \right|^2 - \left|\grad\left| u \right| \right|^2 \rightarrow 0.
\]
\par Step I: \textit{Initial Bounds:}\\
\par Note that
\begin{align*}
\left|\int \left|\grad\left| u+v_n \right| \right|^2 - \left|\grad\left| u \right| \right|^2\right| &\le \int\Big| \left(\grad\left| u+v_n \right| + \grad \left| u \right| \right) \cdot \left( \grad \left| u+v_n \right| - \grad \left| u \right| \right)\Big|\\
& \le \left(2 \norm{\grad u}_2 + \norm{\grad v_n}_2\right) \left(\int \Big| \grad|u+v_n| - \grad |u| \Big|^2\right)^{1/2}.
\end{align*}
To complete the proof it is sufficient to show that
\[\grad|u+v_n| - \grad |u| \overset{L^2}{\rightarrow} 0.\]
\par We split the problem by establishing sufficient bounds with restriction to the sets
\[
E_1^n = \left\{ x \in \mathbb{R}^N: u(x) = 0,\,\,\mbox{or}\,\,u(x)+v_n(x)=0,\,\,\mbox{or}\,\, v_n(x)=0 \right\},
\]
and
\[
E_2^n = \left\{ x \in \mathbb{R}^N: u(x) \neq 0,\,\,\mbox{and}\,\,u(x)+v_n(x) \neq 0,\,\,\mbox{and}\,\,v_n(x) \neq 0\right\}.
\]
\par Step II: \textit{Bounds on $E_1^n$:}\\

\par First consider the case with $u=0$. Note that $\partial_j u = 0 $ a.e. on the set where $u = 0$. Using Kato's inequality we have the a.e. pointwise estimate
\[
\left| \grad|u+v_n| - \grad |u| \right| \le \left| \grad v_n \right|.
\]
\par Now we consider the case with $u \neq 0$ and $u+v_n = 0$. On this set $\grad u = -\grad v_n$ a.e. and so we also have that pointwise a.e.
\begin{align*}
\left| \grad \left|u+v_n \right| - \grad \left| u \right| \right| &= \left| \grad u \right|\\
& \le \left| \grad v_n \right|.
\end{align*}
\par Finally, if $u \neq 0$, $u +v_n \neq 0$ and $v_n = 0$,
\begin{align*}
\left| \grad |u+v_n| - \grad |u| \right| &=
\left|\mbox{Re} \left[\frac{\overline{u+v_n}}{|u+v_n|} \grad (u+v_n) \right] - \mbox{Re} \left[ \frac{\overline{u}}{|u|} \grad u\right]\right|\\\\
& \le \left| \grad v_n \right|
\end{align*}
As $v_n \overset{H^1}{\rightarrow} 0$, it is now sufficient to establish bounds on $E_2^n$.\\

\par Step III: \textit{Setting Up for Convergence on $E_2^n$}\\

\par From here on, we consider $x \in \mathbb{R}^N$ satisfying the conditions of $E_2^n$ with no further comment, and establish bounds for each partial derivative.\\
\par Making using of lemma \ref{lem:ang-dif}, we have the a.e. pointwise estimate
\begin{align*}
\left| \partial_j |u+v_n| - \partial_j |u| \right| &=
\left|\mbox{Re} \left[\frac{\overline{u+v_n}}{|u+v_n|} \partial_j (u+v_n) \right] - \mbox{Re} \left[ \frac{\overline{u}}{|u|} \partial_j u\right]\right|\\\\
& \le \left| \mbox{Re}\left[ \left( \frac{\overline{u+v_n}}{|u+v_n|} - \frac{\overline{u}}{|u|} \right) \partial_j u \right] \right| + \left| \partial_j v_n \right|\\\\
& \le 2\, \mbox{min} \left\{ 1, \frac{|v_n|}{|u|} \right\} \left|\partial_j u \right| + \left| \partial_j v_n \right|.
\end{align*}
Since $v_n \overset{H^1}{\rightarrow} 0$, the problem is reduced to showing that
\[
\mbox{min} \left\{ 1, \frac{|v_n|}{|u|} \right\} \left|\partial_j u \right| \overset{L^2}{\rightarrow} 0.
\]

\par Step IV: \textit{Decomposition Sets}\\

\par We now decompose $E_2^n \subset \mathbb{R}^N$ into sets defined using level sets of $u$ and $v_n$. Consider $M \in \mathbb{Z}$ and define the following sets and functions which are independent of $v_n$.
\begin{align*}
& \Omega_M \coloneqq \left\{ x \in E_2^n : 2^M \le |u(x)| < 2^{M+1}\right\},\\\\
& f_{M}: \mathbb{R}_{\ge 0} \rightarrow \mathbb{R}_{\ge 0}, \quad f_M(\lambda) \coloneqq \mbox{sup}_{\substack{A \subset \Omega_M\\ \mu(A) \le \lambda}} \,\norm{\partial_j u}_{L^2(A)}^2.
\end{align*}
\par Now we define sets which depend on both $u$ and $v_n$. Given $M \in \mathbb{Z}$, consider $L \in \mathbb{Z}$ with $L \le M$, and define
\begin{align*}
& \Omega^{v_n}_{M,L} \coloneqq
\begin{cases}
\left\{ x \in \Omega_M : 2^{L-1} \le |v_n(x)| < 2^L \right\} & L < M\\\\
\left\{ x \in \Omega_M : 2^{L-1} \le |v_n(x)| \right\} & L = M
\end{cases}
\end{align*}

\par Step V \textit{Convergence on Decomposition Sets}\\
\par The definition of $\Omega_{M,L}^{v_n}$ provides the bound
\[
\sum_{M \in \mathbb{Z}} \sum_{\substack{L \in \mathbb{Z} \\ L \le M}} 2^{2L-2} \mu\left( \Omega^{v_n}_{M,L}\right) \le \norm{v_n}_2^2,
\]
and hence
\[
\mu\left( \Omega^{v_n}_{M,L} \right) \le 2^{-2L} \left( 4\norm{v_n}_2^2\right) .
\]
Now,
\begin{align*}
\norm{\mbox{min} \left\{ 1, \frac{|v_n|}{|u|} \right\} \left|\partial_j u \right|}_{L^2( E_2^n)}^2 & \le \sum_{M \in \mathbb{Z}} \sum_{\substack{L \in \mathbb{Z}\\ L \le M}} 2^{2(L-M)} \norm{ \partial_j u}_{L^2 \left( \Omega_{M,L}^{v_n} \right)}^2\\
& \le \sum_{M \in \mathbb{Z}} \sum_{\substack{L \in \mathbb{Z}\\ L \le M}} 2^{2(L-M)} \left[ f_{M} \left( \mu \left( \cup_{\substack{J \in \mathbb{Z}\\L \le J \le M}} \Omega_{M,J}^{v_n} \right)\right) - f_{M} \left( \mu \left( \cup_{\substack{J \in \mathbb{Z}\\L+1 \le J \le M}} \Omega_{M,J}^{v_n} \right)\right) \right]\\
& = \sum_{M \in \mathbb{Z}} \sum_{\substack{L \in \mathbb{Z}\\ L \le M}} 2^{2(L-M)} \left[ f_{M} \left( \sum_{L \le J \le M} \mu \left( \Omega_{M,J}^{v_n} \right)\right) - f_{M} \left( \sum_{L+1 \le J \le M} \mu \left( \Omega_{M,J}^{v_n} \right)\right) \right].
\end{align*}
\par For fixed $M$, the dominated convergence theorem implies that
\[
\sum_{\substack{L \in \mathbb{Z}\\ L \le M}} 2^{2(L-M)} \left[ f_{M} \left( \sum_{L \le J \le M} \mu \left( \Omega_{M,J}^{v_n} \right)\right) - f_{M} \left( \sum_{L+1 \le J \le M} \mu \left( \Omega_{M,J}^{v_n} \right)\right) \right] \rightarrow 0
\]
as $n \rightarrow \infty$. Another application of the dominated convergence theorem implies that
\[
\sum_{M \in \mathbb{Z}} \sum_{\substack{L \in \mathbb{Z}\\ L \le M}} 2^{2(L-M)} \left[ f_{M} \left( \sum_{L \le J \le M} \mu \left( \Omega_{M,J}^{v_n} \right)\right) - f_{M} \left( \sum_{L+1 \le J \le M} \mu \left( \Omega_{M,J}^{v_n} \right)\right) \right] \rightarrow 0
\]
as $n \rightarrow \infty$. Hence there exist sufficient bounds on $\norm{\mbox{min} \left\{ 1, \frac{|v_n|}{|u|} \right\} \left|\partial_j u \right|}_{L^2( E_2^n)}^2$ to imply desired convergence restricted to $E_2^n$\\
\par As noted above, this is sufficient to establish the convergence,
\[
\grad |u+v_n| - \grad | u| \overset{L^2}{\rightarrow} 0,
\]
completing the proof.

\end{proof}

\subsection{Kato Kinetic Energy in NLS}\label{subsec:kato-kinetic-energy-in-nls}
We will now discuss some properties of Kato kinetic energy quantity as they relate to solutions of NLS. We begin by demonstrating an identity for $H^1$ functions which has connections to momentum in NLS. This identity is presented in \cite{CzubakMillerRoudenko24} in the more general setting of covariant derivatives, and is there used to establish dichotomy below the ground state for magnetic nonlinear Schr{\"{o}}dinger equations. Here we include the proof for the more specialized case of the standard gradient.

\begin{lemma}[Remainder Identity for Kato's Inequality]\label{lem:remainder}
For $u \in H^1$, the identity
\begin{align*}
\int_{u \neq 0} \Big| \frac{\mathcal{P}[u]}{2|u|} \Big|^2 dx &= \int_{u \neq 0} \Big| \mbox{Im} \Big[ \frac{\overline{u}}{|u|} \grad u \Big]\Big|^2 dx\\
&= \norm{\grad u}_2^2 - \norm{ \grad |u|}_2^2
\end{align*}
is satisfied.
\end{lemma}
\begin{proof}
Given $u \in H^1$, we have pointwise \cite[Thm 6.17]{LiebLoss01}
\begin{align*}
|\partial_j |u||^2 &= \begin{cases} \left|\mbox{Re} \Big[ \frac{\overline{u}}{|u|} \partial_j u \Big]\right|^2 & u \neq 0\\
0 &u = 0 \end{cases}
\end{align*}
Furthermore by \cite[Thm 6.19]{LiebLoss01} we have that a.e.
\[
\Big| \partial_j u \Big|^2 = \begin{cases} \Big| \frac{\overline{u}}{|u|} \partial_j u \Big|^2 & u \neq 0\\ 0 & u=0. \end{cases}
\]
It follows that a.e.
\[ 
|\grad u|^2 - |\grad |u||^2 =  \begin{cases} \Big|  \mbox{Im} \Big[ \frac{\overline{u}}{|u|} \grad u \Big] \Big|^2 & u \neq 0\\ 0 & u=0.
\end{cases}
\]
From here the result follows by integration.
\end{proof}
\begin{remark}
Pointwise, for $u \neq 0$, we have
\begin{align*}
\left|\im \left[ \frac{\overline{u}}{u} \grad u \right]\right| & \le \left|\frac{\overline{u}}{u} \grad u \right|\\
& =\left| \grad u \right|,
\end{align*}
and so it is natural to interpret $\im \left[ \frac{\overline{u}}{u} \grad u \right]$ as
\begin{equation}\label{eq:interpret}
\im \left[ \frac{\overline{u}}{u} \grad u \right] = \begin{cases} \im \left[ \frac{\overline{u}}{u} \grad u \right]  & u \neq 0 \mbox{ and } \grad u \neq 0\\ 0 & \grad u=0. \end{cases}
\end{equation}
As this is defined a.e. and equal to
\[
\begin{cases} \im \left[ \frac{\overline{u}}{|u|} \grad u \right] & u \neq 0\\ 0 & u=0, \end{cases}
\]
if we take \eqref{eq:interpret} as a definition, it is clear from the proof that the identity of Lemma \ref{lem:remainder} holds a.e. without need for integration.
\end{remark}

This general identity produces a useful inequality related to the setting of NLS.

\begin{lemma}\label{lem:momentum-inequality}
For $u \in H^1$ with $|x|u \in L^2$,
\[\left(\int |x \cdot \frac{\mathcal{P}[u]}{2}| dx\right)^2 + \left(\frac{\textbf{M}[u]}{2}\right)^2 \le \left(\int |x|^2 |u|^2 dx \right) \left( \norm{\grad u}_2^2 - \norm{ \grad |u|}_2^2 \right).\]
\end{lemma}
\begin{proof}
First,
\begin{align*}
\left(\int |x \cdot \frac{\mathcal{P}[u]}{2}| dx\right)^2 &= \left(\int |x||u| \left( \frac{x}{|x|} \cdot \frac{\mathcal{P}[u]}{2|u|} \right) dx\right)^2\\
& \le \left(\int |x|^2|u|^2 dx \right) \left(\int \left|\frac{x}{|x|} \cdot \frac{\mathcal{P}[u]}{2|u|} \right|^2 dx\right).
\end{align*}
Similarly,
\begin{align*}
\left(\int |x \wedge \frac{\mathcal{P}[u]}{2}| dx\right)^2 &\le \left(\int |x|^2|u|^2 dx \right) \left(\int \left|\frac{x}{|x|} \wedge \frac{\mathcal{P}[u]}{2|u|} \right|^2 dx \right).
\end{align*}
By a direct computation,
\[\left| \frac{x}{|x|} \cdot \frac{\mathcal{P}[u]}{2|u|} \right|^2 + \left| \frac{x}{|x|} \wedge \frac{\mathcal{P}[u]}{2|u|} \right|^2 = \left| \frac{\mathcal{P}[u]}{2|u|} \right|^2,\]
and the inequality follows by Lemma \ref{lem:remainder}.
\end{proof}
\begin{remark}\label{rem:kato-statement}
For $u_0$, this inequality can also be written as
\[
\frac{|V_t[u_0]|^2 + |2 \textbf{M}[u_0]|^2}{16 V[u_0]} \le \norm{\grad u_0}_2^2 - \norm{ \grad |u_0|}_2^2,
\]
and each of the energy conditions \eqref{eq:intro-energy-above}, \eqref{eq:bidirectional-energy-assumption}, or \eqref{eq:unidirectional-energy-assumption} imply
\[
\masskatoenergy{u_0} \le \massenergy{u_Q}.
\]
\end{remark}

\subsection{A Nonlinear Inequality Below the Ground State}
\par Our main results for NLS will require a nonlinear inequality which holds which holds when $|u|$ is below the ground state.
\par For this we will introduce the minimal energy function, $P:\mathbb{R}_{\ge 0} \rightarrow \mathbb{R}$ by
\[
P(x) = \frac{1}{2}x  - \frac{C_{GN}}{p+1} x^{\frac{n(p-1)}{4}},
\]
so that the Gagliardo-Nirenberg inequality implies
\[
\massenergy{u} \ge P\left(\masskineticenergy{u}\right).
\]
\par Kenig and Merle \cite{KenigMerle06} used a similar function to establish results for energy critical NLS, and Holmer and Roudenko \cite{HolmerRoudenko07} use such a function in establishing results for intercritical NLS below the ground state. By the Gagliardo-Nirenberg inequality and Kato's inequality, the minimal energy function also satisfies
\begin{equation}\label{eq:minimal-energy-kato}
\masskatoenergy{u} \le P\left(\masskatokineticenergy{u}\right),
\end{equation}
and this property was used in \cite{CzubakMillerRoudenko24} to obtain results for electromagnetic NLS.
\par A computation with \eqref{eq:soliton-identity1} and \eqref{eq:soliton-identity2} shows that $P$ is concave down on $[0,\infty)$, attaining a maximum of $\massenergy{u_Q}$ at $\masskineticenergy{u_Q}$.
\begin{lemma}\label{lem:nonlinear-ineq}
Define $P^*:\bigg[-\masskineticenergy{u_Q},\infty\bigg) \rightarrow \mathbb{R}$ by
\[
P^*(x) = E[u_Q] \norm{u_Q}_2^{2\frac{(1-s_c)}{s_c}} - P \left( x+\norm{\grad u_Q}_2^2 \norm{u_Q}_2^{2\frac{(1-s_c)}{s_c}} \right).
\]
Then we have
\begin{equation}\label{eq:nonlin-def-ineq}
 E[u_Q]\norm{u_Q}_2^{2\frac{(1-s_c)}{s_c}} - \masskatoenergy{u} \le P^*\left(\masskatokineticenergy{u} - \masskineticenergy{u_Q}\right).
 \end{equation}
\end{lemma}
\par This lemma follows immediately by \eqref{eq:minimal-energy-kato} and construction of $P^*$.
\begin{remark}
For the function $P^*$ defined in Lemma \ref{lem:nonlinear-ineq}, in practice we will only be concerned with the restriction $P^*: \mathbb{R}_{\ge 0} \rightarrow \mathbb{R}$. Here $P^*$ is increasing, and satisfies
\[
P^*(0)=\frac{d}{dx} P^* (0)=0.
\]
Furthermore, for this restriction, near $0$ $P^*$ grows like $x^2$, and at infinity $P^*$ grows like $x^{\frac{N(p-1)}{4}}$. Note for intuition, that in the case $s_c \le \frac{N}{4}$, we have $\frac{N(p-1)}{4} \le 2$, and a computation using the second derivative of $P^*(x)$ shows that for $P^*$ restricted to $\mathbb{R}_{\ge 0}$, the inequality \eqref{eq:nonlin-def-ineq} holds (but is weaker) if we replace $P^*(x)$ by a function $F(x)$, defined by
\[
F(x) = \frac{C_{GN}}{p+1} x^{\frac{N(p-1)}{4}}.
\]
\end{remark}

Making use of Lemma \ref{lem:nonlinear-ineq} and \cite[Theorem 3.7]{DuyckaertsRoudenko15}, we obtain the following sufficient condition for scattering.
\begin{corollary}\label{cor:scattering}
Given $u_0 \in H^1$ with
\begin{equation}\label{eq:scattering-kinetic}
\masskatokineticenergy{u_0} < \masskineticenergy{u_Q},
\end{equation}
and corresponding solution $u$ satisfying $u(0) = u_0$, defined on $[0,T^*)$ with $T^*$ being maximal,
if
\begin{equation}\label{eq:scattering-sup}
\sup_{t \in [0,T^*)} \masskatoenergy{u(t)} < \massenergy{u_Q},
\end{equation}
then $T^* = \infty$ and $u$ scatters forward in time in $H^1$.
\end{corollary}
\begin{remark}
We do not need the full power of Lemma \ref{lem:nonlinear-ineq}. It will be apparent that control such as that which appears in \cite[Lemma 3.4]{KenigMerle06} for the energy critical case is sufficient. As such, this is essentially just a rephrasing of \cite[Theorem 3.7]{DuyckaertsRoudenko15} suited to our problem.
\end{remark}
\begin{proof}
If $u_0$ satisfies the hypothesis of Corollary \ref{cor:scattering}, then \eqref{eq:scattering-kinetic} and \eqref{eq:scattering-sup} imply that \eqref{eq:scattering-kinetic} holds for the full positive interval of existence, and $T^*= \infty$. With Lemma \ref{lem:nonlinear-ineq}, this implies
\begin{equation}\notag
\sup_{t \in [0,\infty)} P^* \left( \masskatokineticenergy{u_Q}- \masskatokineticenergy{u(t)} \right)> 0,
\end{equation}
and hence
\begin{equation}\notag
\sup_{t \in [0,\infty)} \masskatokineticenergy{u(t)} < \masskatokineticenergy{u_Q}.
\end{equation}
\par Using
\begin{align*}
&\norm{w}_{p+1}^{p+1} \norm{w}_2^{2 \frac{(1-s_c)}{s_c}} - \norm{u_Q}_{p+1}^{p+1} \norm{u_Q}_2^{2 \frac{(1-s_c)}{s_c}}\\
&= \frac{p+1}{2} \left( \masskatokineticenergy{w} - \masskineticenergy{u_Q} \right) + (p+1)\left(\massenergy{u_Q} - \masskatoenergy{w} \right)\\
& \ge \frac{p+1}{2} \left( \masskatokineticenergy{w} - \masskineticenergy{u_Q} \right),
\end{align*}
we can directly apply \cite[Theorem 3.7]{DuyckaertsRoudenko15}, and $u$ scatters forward in time in $H^1$.
\end{proof}

\section{Global Existence and Blow-Up}
In this section we prove the main results regarding global existence and blow-up for solutions to NLS.

\subsection{Proof of Theorem \ref{thm:main-unidirectional}}

\par Here we prove \ref{thm:main-unidirectional}, and establish results in one time direction.
\begin{proof}
\par In this proof we will assume \eqref{eq:unidirectional-lower-energy-assumption} without further comment. As noted in Remark \ref{rem:kato-statement}, if $u$ satisfies \eqref{eq:unidirectional-energy-assumption}, then $u$ also satisfies
\begin{equation}\label{eq:unidirectional-kato-statement}
\masskatoenergy{u} \le \massenergy{u_Q}.
\end{equation}
If \eqref{eq:unidirectional-energy-assumption} is strict, then \eqref{eq:unidirectional-kato-statement} is strict as well.
\par We begin with the case of global existence, and introduce a set of assumption which is stronger than what appears in the hypothesis of the theorem.
\begin{assumption}\label{ass:un-ex-end}
We say $w \in H^1$ satisfies assumption \ref{ass:un-ex-end} if $w$ satisfies (with $u$ or $u_0$ replaced with $w$ where necessary).
\begin{align}
& - \quad \quad \left[ E[w] - \frac{|V_t[w]|^2 + |2 \textbf{M}[w]|^2}{32 V[w]}\right] \norm{w}_2^{2\frac{(1-s_c)}{s_c}} <  E[w] \norm{w}_2^{2\frac{(1-s_c)}{s_c}} \label{eq:unidirectional-energy-assumption-strict},\\\notag\\
& - \quad \quad \eqref{eq:kinetic-existence-uni},\,\mbox{and} \notag\\\notag\\
& - \quad \quad V_t[w] > 0\label{eq:variance-existence-strict}.
\end{align}
\end{assumption}
\par We will establish the following with $u$ as the solution with initial data $u_0$.
\begin{enumerate}
\item If $u_0$ satisfies \eqref{eq:unidirectional-energy-assumption}, \eqref{eq:kinetic-existence-uni} and \eqref{eq:variance-existence-uni}, then there is some $\epsilon>0$ so that on the open interval $(0,\epsilon)$, $u$ satisfies assumption \ref{ass:un-ex-end}.
\item If $u(T)$ satisfies assumption \ref{ass:un-ex-end}, then there is some $\epsilon>0$ so that on the interval $[T,T+\epsilon)$, $u$ satisfies assumption \ref{ass:un-ex-end}.
\item If $u$ satisfies assumption \ref{ass:un-ex-end} on an open interval $(T_1,T_2)$, then $u$ can be extended to $(T_1,T_2]$, and $u(T_2)$ satisfies assumption \ref{ass:un-ex-end}.
\end{enumerate}
\par If (1), (2), and (3) hold, then global existence follows immediately by a bootstrap argument. It is sufficient to prove (1) and (3).\\
\par Proof of (1): Assume $u_0$ satisfies \eqref{eq:unidirectional-energy-assumption}, \eqref{eq:kinetic-existence-uni}, and \eqref{eq:variance-existence-uni}. Using \eqref{eq:variance-value}, since $u_0$ satisfies \eqref{eq:unidirectional-lower-energy-assumption}, \eqref{eq:unidirectional-energy-assumption}, and \eqref{eq:kinetic-existence-uni},
\begin{align}
V_{tt}[u_0] &= 8\norm{\grad u_0}_2^2 - \frac{4N(p-1)}{p+1} \norm{u_0}_{p+1}^{p+1}\label{eq:virial-unidirectional-existence-bs}\\
&> 8\left(\norm{\grad u_0}_2^2 - \norm{\grad |u_0|}_2^2 \right)\notag\\
&> 0 \notag.
\end{align}
In the sequel, we will make use of the intermediate inequality as well as the final inequality of \eqref{eq:virial-unidirectional-existence-bs}.
\par We compute
\begin{align}\label{eq:varcomp1}
&\partial_t \left[ \frac{ \left|V_t[u_0]\right|^2 + \left| 2 \textbf{M}[u_0] \right|^2}{ V[u_0]} \right]\\
& = \frac{V_t[u_0]}{V[u_0]} \left[ 2 V_{tt} [u_0] - \left( \frac{ \left| V_t [u_0]\right|^2 + \left|2 \textbf{M}[u_0] \right|^2}{V[u_0]}\right) \right]\notag.
\end{align}
Using \eqref{eq:virial-unidirectional-existence-bs} and Lemma \ref{lem:momentum-inequality},
\begin{equation}\label{eq:unidirectional-exist-initial-gap}
\left[ 2 V_{tt} [u_0] - \left( \frac{ \left| V_t [u_0]\right|^2 + \left|2 \textbf{M}[u_0] \right|^2}{V[u_0]}\right) \right] >0.
\end{equation}
\par Now, we can choose sufficiently small $\epsilon$ so that on $(0,\epsilon)$, using \eqref{eq:variance-existence-uni} and \eqref{eq:virial-unidirectional-existence-bs}, $u$ satisfies \eqref{eq:variance-existence-strict}. Using strict inequality in \eqref{eq:kinetic-existence-uni}, with $\epsilon$ chosen sufficiently small, $u$ also satisfies \eqref{eq:kinetic-existence-uni}. Using strict inequality in both \eqref{eq:unidirectional-exist-initial-gap} and \eqref{eq:virial-unidirectional-existence-bs}, for $\epsilon$ sufficiently small, $u$ further satisfies
\begin{equation}\label{eq:interval-gap-drop}
\partial_t \left[ \frac{ \left|V_t[u]\right|^2 + \left| 2 \textbf{M}[u] \right|^2}{V[u]} \right] >0.
\end{equation}
Together with \eqref{eq:unidirectional-energy-assumption}, we have that $u$ satisfies \eqref{eq:unidirectional-energy-assumption-strict}. Hence, $u$ satisfies assumption \ref{ass:un-ex-end} on $(0,\epsilon)$ and the proof of (1) is complete.\\
\par Proof of (3): Assume that $u$ satisfies assumption \ref{ass:un-ex-end} on $(T_1,T_2)$. Using \eqref{eq:kinetic-existence-uni}, we can extend $u$ to $(T_1,T_2]$. As in the proof of (1), on $(T_1,T_2)$ $u$ satisfies \eqref{eq:virial-unidirectional-existence-bs} and \eqref{eq:interval-gap-drop}. Since $u$ satisfies these as well as \eqref{eq:unidirectional-energy-assumption-strict} and \eqref{eq:variance-existence-strict} on $(T_1,T_2)$, $u(T_2)$ also satisfies \eqref{eq:unidirectional-energy-assumption-strict} and \eqref{eq:variance-existence-strict}. Since $u$ satisfies \eqref{eq:kinetic-existence-uni} on $(T_1,T_2)$, $u(T_2)$ satisfies
\[
\masskatokineticenergy{u(T_2)} \le \masskineticenergy{u_Q}.
\]
Furthermore, since \eqref{eq:unidirectional-energy-assumption-strict} implies the strict inequality $\masskatoenergy{u(T_2)} < \massenergy{u_Q}$, $u(T_2)$ satisfies \eqref{eq:kinetic-existence-uni}. Hence $u(T_2)$ satisfies assumption \ref{ass:un-ex-end} and the proof of (3) is complete.\\
\par This completes the proof of global existence. Scattering follows from \eqref{eq:unidirectional-energy-assumption-strict} at $u(\epsilon)$, \eqref{eq:interval-gap-drop} on $[\epsilon,\infty)$, and Corollary \ref{cor:scattering}.
\par We now move on to the case of blow-up. for ease in the proof, we begin by restating the hypothesis of the theorem for blow-up.
\begin{assumption}\label{ass:un-bl-end}
We say $w \in H^1$ satisfies assumption \ref{ass:un-bl-end} if $w$ satisfies all of \eqref{eq:unidirectional-energy-assumption}, \eqref{eq:kinetic-blow-up-uni}, \eqref{eq:variance-blow-up-uni}, and \eqref{eq:unidirectional-nonlinear-condition}, where for each condition $u$ or $u_0$ is replaced with $w$ as necessary necessary. In particular, \eqref{eq:unidirectional-nonlinear-condition} becomes
\begin{align}
    \frac{|2\textbf{M}[w]|^2}{32 V[w]}\norm{w}_2^{2\frac{(1-s_c)}{s_c}} > F \left( \left(E[w] - \frac{|V_t[w]|^2}{32 V[w]}\right)\norm{w}_2^{2\frac{(1-s_c)}{s_c}} - \massenergy{u_Q} \right).
    \end{align}
\end{assumption}
By another bootstrap argument, to establish the blow-up result it is sufficient to show the following.
\begin{enumerate}
\item If $u(T)$ satisfies assumption \ref{ass:un-bl-end}, then there is some $\epsilon>0$ so that on the interval $[T,T+\epsilon)$, $u$ satisfies assumption \ref{ass:un-bl-end}.
\item If $u$ is defined on an interval $(T_1,T_2]$, and $u$ satisfies assumption \ref{ass:un-bl-end} on the open interval $(T_1,T_2)$, then $u(T_2)$ satisfies assumption \ref{ass:un-bl-end}.
\item If there is some $T$ so for any $t$ in the interval of existence of $u$ with $t>T$, $u(t)$ satisfies assumption \ref{ass:un-bl-end}, then $u$ blows-up in finite positive time.
\end{enumerate}
\par To complete the theorem we must establish (1), (2), and (3).\\
\par Before we establish (1), (2), and (3), we perform some computations. Suppose some fixed $w$ satisfies assumption \ref{ass:un-bl-end}. Then 
\[
\frac{V_t[w]}{ V[w]} <0,
\]
and using \eqref{eq:kinetic-blow-up-uni},
\begin{align*}
&\left(32 E[w] - \frac{ \left| V_t [w]\right|^2 +|2\textbf{M}[w]|^2}{V[w]} - \frac{8(N(p-1) - 4)}{p+1} \norm{w}_{p+1}^{p+1}\right) \norm{w}_2^{2\frac{(1-s_c)}{s_c}}\\
& \le 32 \massenergy{u_Q} - \frac{8(N(p-1) - 4)}{p+1} \norm{w}_{p+1}^{p+1}\norm{w}_2^{2 \frac{(1-s_c)}{s_c}}\\
& < 0.
\end{align*}
By computation, as with \eqref{eq:varcomp1}, we have
\begin{equation}\label{eq:gap-deriv-uni-2}
\partial_t \left[ \frac{ \left|V_t[w]\right|^2 + \left| 2 \textbf{M}[w] \right|^2}{32 V[w]} \right] > 0.
\end{equation}
\par We continue by computing 
\begin{align*}
&\partial_t \left[ \frac{ \left|V_t[w]\right|^2 }{ V[w]} \right] \norm{w}_2^{2\frac{(1-s_c)}{s_c}}\\
& \quad = \frac{V_t[w]}{ V[w]} \left[ 32 E[w] - \frac{ \left| V_t [w]\right|^2 }{V[w]} - \frac{8(N(p-1) - 4)}{p+1} \norm{w}_{p+1}^{p+1} \right] \norm{w}_2^{2 \frac{(1-s_c)}{s_c}}.
\end{align*}
Note that \eqref{eq:unidirectional-energy-assumption} and \eqref{eq:kinetic-blow-up-uni} imply $\masskatoenergy{w} \le \massenergy{u_Q}$ and together with \eqref{eq:unidirectional-nonlinear-condition}, Lemma \ref{lem:momentum-inequality}, and Lemma \ref{lem:nonlinear-ineq}, we have
\begin{align*}
&\norm{w}_{p+1}^{p+1} \norm{w}_2^{2 \frac{(1-s_c)}{s_c}} - \norm{u_Q}_{p+1}^{p+1} \norm{u_Q}_2^{2 \frac{(1-s_c)}{s_c}}\\
&= \frac{p+1}{2} \left( \masskatokineticenergy{w} - \masskineticenergy{u_Q} \right) + (p+1)\left(\massenergy{u_Q} - \masskatoenergy{w} \right)\\
& \ge \frac{p+1}{2} \left( \masskatokineticenergy{w} - \masskineticenergy{u_Q} \right)\\
& > \frac{4(p+1)}{N(p-1) - 4} \left( \left(E[w] - \frac{|V_t[w]|^2}{32 V[w]}\right)\norm{w}_2^{2\frac{(1-s_c)}{s_c}} - E[u_Q]\norm{u_Q}_2^{2\frac{(1-s_c)}{s_c}} \right).
\end{align*}
Hence,
\begin{align*}
& \left[ 32 E[w] - \frac{ \left| V_t [w]\right|^2 }{V[w]} - \frac{8(N(p-1) - 4)}{p+1} \norm{w}_{p+1}^{p+1} \right] \norm{w}_2^{2 \frac{(1-s_c)}{s_c}}\\
& =  \left(32 E[w] - \frac{ \left| V_t [w]\right|^2 }{V[w]}\right) \norm{u}_2^{2 \frac{(1-s_c)}{s_c}} - 32\massenergy{u_Q}\\
& \quad - \frac{8(N(p-1) - 4)}{p+1}\left( \norm{w}_{p+1}^{p+1}\norm{w}_2^{2 \frac{(1-s_c)}{s_c}}- \norm{u_Q}_{p+1}^{p+1}\norm{u_Q}_2^{2 \frac{(1-s_c)}{s_c}}\right)  \\
& < 0,
\end{align*}
and so
\begin{equation}\label{eq:unidirectional-gap-drop-2}
\partial_t \left[ \frac{ \left|V_t[w]\right|^2 }{ V[w]} \right]>0.
\end{equation}
Finally, note
\begin{align*}
\partial_t \left[ \sqrt{V[w]}\right] &= \frac{V_t[w]}{2 \sqrt{V[w]}}\\
&= - \frac{1}{2} \left( \frac{|V_t[w]|^2}{V[w]} \right)^{\frac{1}{2}}
\end{align*}
\par Proof of (1): Application of \eqref{eq:gap-deriv-uni-2}, along with strict inequalities implies that (1) holds.
\par Proof of (2): Assume $u$ is defined on some interval $(T_1,T_2]$, and that assumption \ref{ass:un-bl-end} holds on $(T_1,T_2)$. We immediately have that $u(T_2)$ satisfies \eqref{eq:unidirectional-energy-assumption}. In fact, by \eqref{eq:gap-deriv-uni-2}, we have that on $(T_1,T_2]$, $u(t)$ satisfies \eqref{eq:unidirectional-energy-assumption} with strict inequality. This strict inequality implies that $u(T_2)$ satisfies \eqref{eq:kinetic-blow-up-uni}. Furthermore, \eqref{eq:unidirectional-gap-drop-2} and \eqref{eq:variance-blow-up-uni} on $(T_1,T_2)$ and the fact that $u$ does not blow-up at $T_2$ implies by continuity arguments that $V_t[u(T_2)] \neq 0$, and in particular, $u(T_2)$ satisfies \eqref{eq:variance-blow-up-uni}. To see that $u(T_2)$ satisfies \eqref{eq:unidirectional-nonlinear-condition}, note that $F$ is an increasing function, and \eqref{eq:variance-blow-up-uni} implies
\[
\partial_t \left[\frac{|2\textbf{M[u]}|^2}{V[u]}\right] >0,
\]
while \eqref{eq:unidirectional-gap-drop-2} implies
\[
\partial_t\left[  \left(E[w] - \frac{|V_t[w]|^2}{32 V[w]}\right)\norm{w}_2^{2\frac{(1-s_c)}{s_c}} - E[u_Q]\norm{u_Q}_2^{2\frac{(1-s_c)}{s_c}} \right] <0.
\]
Hence $u(T_2)$ satisfies \eqref{eq:unidirectional-nonlinear-condition}, and (2) holds.
\par Proof of (3): Without loss of generality, assume $u$ is defined at $t=0$, and satisfies assumption \ref{ass:un-bl-end} for all positive time in its interval of existence. Since
\[
\partial_t \left[ \frac{ \left|V_t[u](t)\right|^2}{32 V[u](t)} \right] > 0,
\]
and
\[
V_t[u](t)< 0,
\]
we have
\begin{align*}
& \partial_t \left[ \frac{V_t[u](t)}{\sqrt{V[u](t)}} \right] < 0,\\
&\frac{V_t[u](t)}{\sqrt{V[u](t)} }< \frac{V_t[u](0)}{\sqrt{V[u](0)}},
\end{align*}
and
\[
\partial_t \left[\sqrt{V[u](t)}\right] < \frac{V_t[u](0)}{2\sqrt{V[u](0)}}<0.
\]
The nonnegativity of $\sqrt{V[u](t)}$ implies that $u$ blows-up in finite positive time by a contradiction argument (or by the uncertainty principle). This completes the proof of blow-up.
\end{proof}

\subsection{Proof of Theorem \ref{thm:main-bidirectional}}
We will now prove \ref{thm:main-bidirectional}, establishing global existence and blow-up in both time directions.
\begin{proof}
\par We establish the result forward in time under conditions which are invariant under time reversal. Results in both time directions follow from time reversal. Throughout the proof we will assume \eqref{eq:unidirectional-lower-energy-assumption} without further comment.
\par We begin by showing global existence. Assume $u_0$ is initial data satisfying \eqref{eq:bidirectional-energy-assumption} and \eqref{eq:kinetic-existence-bi}.
\par If $V_t[u](0) \ge 0$, then the result follows by Theorem \ref{thm:main-unidirectional}.
\par If $V_t[u](0) <0$, then by strict inequality, there is a sufficiently small interval $[0,T)$ so that on $[0,T)$, $u$ satisfies $V_t[u](t)<0$, and hence \eqref{eq:bidirectional-energy-assumption}, and also satisfies \eqref{eq:kinetic-existence-bi}. Now suppose $[0,S)$ is any interval on which a solution satisfies \eqref{eq:bidirectional-energy-assumption} and \eqref{eq:kinetic-existence-bi}. By \eqref{eq:kinetic-existence-bi}, $u$ can be extended to $[0,S]$. Either there is some $t \in [0,S]$ such that $V_t[u](0) \ge 0$, or else $V_t[u](0) < 0$ on all of $[0,S)$ and
\begin{equation}\label{eq:bi-interval-drop}
\frac{|2 \textbf{M}[u_0]|^2}{32 V[u](T)} > \frac{|2 \textbf{M}[u_0]|^2}{32 V[u](0)},
\end{equation}
so \eqref{eq:bidirectional-energy-assumption} holds with strict inequality. This implies that \eqref{eq:kinetic-existence-bi} also holds for $u(T)$ and global existence follows by bootstrap. Consider $u$ on $[\epsilon,\infty)$, where \eqref{eq:bidirectional-energy-assumption} holds with strict inequality for $u(\epsilon)$. An application of \eqref{eq:bi-interval-drop} implies that for $t>\epsilon$,
\begin{equation}\notag
\masskatoenergy{u(t)} < \masskatoenergy{u(\epsilon)},
\end{equation}
which, applying Corollary \ref{cor:scattering}, shows that $u$ scatters forward in time.\\
\par We now prove the blow-up result.
\par We assume $u_0$ satisfies \eqref{eq:bidirectional-energy-assumption}, and \eqref{eq:kinetic-blow-up-bi}, and we will take $\lambda>1$ to be fixed for the remainder of the proof. If $V_t[u_0] < 0$ and $u_0$ satisfies \eqref{eq:bidirectional-nonlinear-condition-2}, then the proof is complete by application of Theorem \ref{thm:main-unidirectional}. If $V_t[u_0] =0$ and $u_0$ satisfies \eqref{eq:bidirectional-nonlinear-condition-1}, then \eqref{eq:bidirectional-nonlinear-condition-1} implies that $V_{tt}[u_0] < 0$, and by strict inequalities of the conditions there is some small time where $u(t)$ satisfies the blow-up conditions of Theorem \ref{thm:main-unidirectional}. Finally, assume that $V_t[u_0]>0$ and that $u_0$ satisfies \eqref{eq:bidirectional-nonlinear-condition-2}. For the remainder of the proof, we will always assume $u$ is a solution such that $u_0$ satisfies \eqref{eq:bidirectional-energy-assumption}, \eqref{eq:kinetic-blow-up-bi}, and \eqref{eq:bidirectional-nonlinear-condition-2} with $V_t[u_0]>0$.
\par We consider the following time dependent assumption.
\begin{assumption}\label{ass:bi-bootstrap}
Given $u(t)$ for some $t \in [0,(\lambda-1)\frac{V[u_0]}{V_t[u_0]})$, we say that $u(t)$ satisfies assumption \ref{ass:bi-bootstrap} if $u(t)$ satisfies
\begin{align}
& - \quad \quad \eqref{eq:kinetic-blow-up-bi}\notag\\\notag\\
& - \quad \quad V[u](t) \le V[u_0] + t V_t[u_0],\label{eq:bi-bootstrap-bl-2}\\\notag\\
& - \quad \quad V_t[u](t) \le V_t[u_0] - t \frac{|V_t[u_0]|^2}{(\lambda-1)V[u_0]},\label{eq:bi-bootstrap-bl-3}
\end{align}
\end{assumption}
\par This assumption holds trivially for $u_0$. Notice that if $u\left( (\lambda - 1) \frac{V[u_0]}{V_t[u_0]}\right)$ satisfies Assumption \ref{ass:bi-bootstrap}, then $u\left( (\lambda - 1) \frac{V[u_0]}{V_t[u_0]}\right)$ satisfies \eqref{eq:bidirectional-energy-assumption}, \eqref{eq:kinetic-blow-up-bi}, \eqref{eq:bidirectional-nonlinear-condition-1}, and $V_t\left[ u\left( (\lambda - 1) \frac{V[u_0]}{V_t[u_0]}\right) \right] =0$, which we have already demonstrated imply blow-up. So in order to establish blow-up for $u$, it is sufficient to show that the following hold.
\begin{enumerate}
\item If, for some $0 \le T < (\lambda - 1) \frac{V[u_0]}{V_t[u_0]}$, $u(T)$ satisfies assumption \ref{ass:bi-bootstrap}, then there is some $\epsilon>0$ so that $u(t)$ satisfies \ref{ass:bi-bootstrap} on $[T,T+\epsilon)$.
\item If, for some $0<S \le (\lambda-1)\frac{|V_t[u_0]|^2}{V[u_0]}$, $u$ is defined on $[0,S]$ and satisfies assumption \ref{ass:bi-bootstrap} on $[0,S)$, then $u(S)$ satisfies assumption \ref{ass:bi-bootstrap}.
\end{enumerate}
\par For (2), the only property which needs to be shown is \eqref{eq:kinetic-blow-up-bi}, which follows from \eqref{eq:bi-bootstrap-bl-2} and strict inequality
\begin{equation}\notag
\left[ E[u_0] - \frac{|2 \textbf{M}[u_0]|^2}{32 \lambda V[u_0]}\right] \norm{u_0}_2^{2\frac{(1-s_c)}{s_c}} < E[u_Q]\norm{u_Q}_2^{2\frac{(1-s_c)}{s_c}}.
\end{equation}
Assume $u(T)$ satisfies Assumption \ref{ass:bi-bootstrap} at some $0 \le T < (\lambda - 1) \frac{V[u_0]}{V_t[u_0]}$. Then $V[u](T)< \lambda V[u_0]$, which, with \eqref{eq:bidirectional-nonlinear-condition-2}, implies 
\begin{align*}
\frac{|2\textbf{M}[u_0]|^2}{32 V[u](t)}\norm{u_0}_2^{2\frac{(1-s_c)}{s_c}} &\ge \frac{|2\textbf{M}[u_0]|^2}{32 \lambda V[u_0]}\norm{u_0}_2^{2\frac{(1-s_c)}{s_c}}\\
&> G \left(E[u_0]\norm{u_0}_2^{2\frac{(1-s_c)}{s_c}} - E[u_Q]\norm{u_Q}_2^{2\frac{(1-s_c)}{s_c}}, \frac{|V_t[0]|^2}{(\lambda-1) V[0]}  \norm{u_0}_2^{2\frac{(1-s_c)}{s_c}}\right),
\end{align*}
and hence, the virial identity and use of Lemma \ref{lem:nonlinear-ineq} as in the proof of Theorem \ref{thm:main-unidirectional},
\begin{equation}\notag
V_{tt} < - \frac{|V_t[u_0]|^2}{(\lambda-1)V[u_0]}.
\end{equation}
This is sufficient to imply we can choose $\epsilon>0$ so that $u$ satisfies Assumption \ref{ass:bi-bootstrap} on $[T,T+\epsilon)$. Hence (1) holds and $u$ blows-up in finite time.
\end{proof}
\nocite{*}

\bibliographystyle{plain}
\bibliography{bibliography}

\begin{thebibliography}{10}

\bibitem{Beceanu08}
Marius Beceanu.
\newblock A centre-stable manifold for the focussing cubic {NLS} in {$\Bbb
  R^{1+3}$}.
\newblock {\em Comm. Math. Phys.}, 280(1):145--205, 2008.

\bibitem{BerestyckiLions83}
H.~Berestycki and P.-L. Lions.
\newblock Nonlinear scalar field equations. {I}. {E}xistence of a ground state.
\newblock {\em Arch. Rational Mech. Anal.}, 82(4):313--345, 1983.

\bibitem{CamposFarahRoudenko22}
Luccas Campos, Luiz~Gustavo Farah, and Svetlana Roudenko.
\newblock Threshold solutions for the nonlinear {S}chr\"odinger equation.
\newblock {\em Rev. Mat. Iberoam.}, 38(5):1637--1708, 2022.

\bibitem{Cazenave03}
Thierry Cazenave.
\newblock {\em Semilinear {S}chr\"odinger equations}, volume~10 of {\em Courant
  Lecture Notes in Mathematics}.
\newblock New York University, Courant Institute of Mathematical Sciences, New
  York; American Mathematical Society, Providence, RI, 2003.

\bibitem{CazenaveWeissler92}
Thierry Cazenave and Fred~B. Weissler.
\newblock Rapidly decaying solutions of the nonlinear {S}chr\"odinger equation.
\newblock {\em Comm. Math. Phys.}, 147(1):75--100, 1992.

\bibitem{CzubakMillerRoudenko24}
Magdalena Czubak, Ian Miller, and Svetlana Roudenko.
\newblock Dichotomy results for the electromagnetic schr{\"{o}}dinger
  equations, 2024.
\newblock preprint.

\bibitem{Dinh20}
Van~Duong Dinh.
\newblock A unified approach for energy scattering for focusing nonlinear
  {S}chr\"odinger equations.
\newblock {\em Discrete Contin. Dyn. Syst.}, 40(11):6441--6471, 2020.

\bibitem{DuyckaertsHolmerRoudenko08}
Thomas Duyckaerts, Justin Holmer, and Svetlana Roudenko.
\newblock Scattering for the non-radial 3{D} cubic nonlinear {S}chr\"odinger
  equation.
\newblock {\em Math. Res. Lett.}, 15(6):1233--1250, 2008.

\bibitem{DuyckaertsRoudenko10}
Thomas Duyckaerts and Svetlana Roudenko.
\newblock Threshold solutions for the focusing 3{D} cubic {S}chr\"odinger
  equation.
\newblock {\em Rev. Mat. Iberoam.}, 26(1):1--56, 2010.

\bibitem{DuyckaertsRoudenko15}
Thomas Duyckaerts and Svetlana Roudenko.
\newblock Going beyond the threshold: scattering and blow-up in the focusing
  {NLS} equation.
\newblock {\em Comm. Math. Phys.}, 334(3):1573--1615, 2015.

\bibitem{GinibreVelo79-1}
J.~Ginibre and G.~Velo.
\newblock On a class of nonlinear {S}chr\"odinger equations. {I}. {T}he
  {C}auchy problem, general case.
\newblock {\em J. Functional Analysis}, 32(1):1--32, 1979.

\bibitem{Glassey77}
R.~T. Glassey.
\newblock On the blowing up of solutions to the {C}auchy problem for nonlinear
  {S}chr\"odinger equations.
\newblock {\em J. Math. Phys.}, 18(9):1794--1797, 1977.

\bibitem{GuevaraCarreon12}
Cristi Guevara and Fernando Carreon.
\newblock Scattering and blow up for the two-dimensional focusing quintic
  nonlinear {S}chr\"odinger equation.
\newblock In {\em Recent advances in harmonic analysis and partial differential
  equations}, volume 581 of {\em Contemp. Math.}, pages 117--153. Amer. Math.
  Soc., Providence, RI, 2012.

\bibitem{Guevara14}
Cristi~Darley Guevara.
\newblock Global behavior of finite energy solutions to the {$d$}-dimensional
  focusing nonlinear {S}chr\"odinger equation.
\newblock {\em Appl. Math. Res. Express. AMRX}, (2):177--243, 2014.

\bibitem{HolmerPlatteRoudenko10}
Justin Holmer, Rodrigo Platte, and Svetlana Roudenko.
\newblock Blow-up criteria for the 3{D} cubic nonlinear {S}chr\"odinger
  equation.
\newblock {\em Nonlinearity}, 23(4):977--1030, 2010.

\bibitem{HolmerRoudenko07}
Justin Holmer and Svetlana Roudenko.
\newblock On blow-up solutions to the 3{D} cubic nonlinear {S}chr\"odinger
  equation.
\newblock {\em Appl. Math. Res. Express. AMRX}, pages Art. ID abm004, 31, 2007.
\newblock [Issue information previously given as no. 1 (2007)].

\bibitem{HolmerRoudenko08}
Justin Holmer and Svetlana Roudenko.
\newblock A sharp condition for scattering of the radial 3{D} cubic nonlinear
  {S}chr\"odinger equation.
\newblock {\em Comm. Math. Phys.}, 282(2):435--467, 2008.

\bibitem{HolmerRoudenko10}
Justin Holmer and Svetlana Roudenko.
\newblock Divergence of infinite-variance nonradial solutions to the 3{D} {NLS}
  equation.
\newblock {\em Comm. Partial Differential Equations}, 35(5):878--905, 2010.

\bibitem{KenigMerle06}
Carlos~E. Kenig and Frank Merle.
\newblock Global well-posedness, scattering and blow-up for the
  energy-critical, focusing, non-linear {S}chr\"odinger equation in the radial
  case.
\newblock {\em Invent. Math.}, 166(3):645--675, 2006.

\bibitem{LiebLoss01}
Elliott~H. Lieb and Michael Loss.
\newblock {\em Analysis}, volume~14 of {\em Graduate Studies in Mathematics}.
\newblock American Mathematical Society, Providence, RI, second edition, 2001.

\bibitem{NakanishiSchlag12}
K.~Nakanishi and W.~Schlag.
\newblock Global dynamics above the ground state energy for the cubic {NLS}
  equation in 3{D}.
\newblock {\em Calc. Var. Partial Differential Equations}, 44(1-2):1--45, 2012.

\bibitem{SulemSulem99}
Catherine Sulem and Pierre-Louis Sulem.
\newblock {\em The nonlinear {S}chr\"odinger equation}, volume 139 of {\em
  Applied Mathematical Sciences}.
\newblock Springer-Verlag, New York, 1999.
\newblock Self-focusing and wave collapse.

\bibitem{Tao06}
Terence Tao.
\newblock {\em Nonlinear dispersive equations}, volume 106 of {\em CBMS
  Regional Conference Series in Mathematics}.
\newblock Conference Board of the Mathematical Sciences, Washington, DC; by the
  American Mathematical Society, Providence, RI, 2006.
\newblock Local and global analysis.

\bibitem{Weinstein82}
Michael~I. Weinstein.
\newblock Nonlinear {S}chr\"odinger equations and sharp interpolation
  estimates.
\newblock {\em Comm. Math. Phys.}, 87(4):567--576, 1982/83.

\end{thebibliography}

\end{document}